\documentclass[10 pt]{amsart}
\usepackage{amssymb,amsmath,epsfig,mathrsfs, enumerate}
\usepackage{amssymb,amsmath,epsfig,graphics,mathrsfs,enumerate,verbatim}
\usepackage[pagebackref,colorlinks=true,linkcolor=blue,citecolor=blue]{hyperref}

\usepackage{graphicx}
\usepackage[normalem]{ulem}
\usepackage{fancyhdr}
\usepackage[margin=3cm]{geometry}
\usepackage{hyperref}
\hypersetup{
 colorlinks   = true,
 urlcolor     = blue,
 linkcolor    = blue,
 citecolor   = red ,
 bookmarksopen=true
}

\theoremstyle{plain}
\newtheorem{thrm}{Theorem}[section]
\newtheorem{lemma}[thrm]{Lemma}
\newtheorem{prop}[thrm]{Proposition}

\newtheorem{rmrk}[thrm]{Remark}
\newtheorem{dfn}[thrm]{Definition}

\begin{document}

\newcommand{\SL}{\mathcal L^{1,p}( D)}
\newcommand{\Lp}{L^p( Dega)}
\newcommand{\CO}{C^\infty_0( \Omega)}
\newcommand{\Rn}{\mathbb R^n}
\newcommand{\Rm}{\mathbb R^m}
\newcommand{\R}{\mathbb R}
\newcommand{\Om}{\Omega}
\newcommand{\Hn}{\mathbb H^n}
\newcommand{\aB}{\alpha B}
\newcommand{\eps}{\ve}
\newcommand{\BVX}{BV_X(\Omega)}
\newcommand{\p}{\partial}
\newcommand{\IO}{\int_\Omega}
\newcommand{\bG}{\boldsymbol{G}}
\newcommand{\bg}{\mathfrak g}
\newcommand{\bz}{\mathfrak z}
\newcommand{\bv}{\mathfrak v}
\newcommand{\Bux}{\mbox{Box}}
\newcommand{\e}{\ve}
\newcommand{\X}{\mathcal X}
\newcommand{\Y}{\mathcal Y}
\newcommand{\W}{\mathcal W}
\newcommand{\la}{\lambda}
\newcommand{\vf}{\varphi}
\newcommand{\rhh}{|\nabla_H \rho|}
\newcommand{\Za}{Z_\beta}
\newcommand{\ra}{\rho_\beta}
\newcommand{\na}{\nabla_\beta}
\newcommand{\vt}{\vartheta}
\newcommand{\La}{\mathcal{L}}
\newcommand{\f}{\ell}

\newcommand{\di}{\operatorname{div}}

\numberwithin{equation}{section}

\newcommand{\RN} {\mathbb{R}^N}
\newcommand{\Sob}{S^{1,p}(\Omega)}
\newcommand{\Dxk}{\frac{\partial}{\partial x_k}}
\newcommand{\Co}{C^\infty_0(\Omega)}
\newcommand{\Je}{J_\ve}
\newcommand{\beq}{\begin{equation}}
\newcommand{\bea}[1]{\begin{array}{#1} }
\newcommand{\eeq}{ \end{equation}}
\newcommand{\ea}{ \end{array}}
\newcommand{\eh}{\ve h}
\newcommand{\D}{\operatorname{div}}
\newcommand{\Dxi}{\frac{\partial}{\partial x_{i}}}
\newcommand{\Dyi}{\frac{\partial}{\partial y_{i}}}
\newcommand{\Dt}{\frac{\partial}{\partial t}}
\newcommand{\aBa}{(\alpha+1)B}
\newcommand{\GF}{^{1+\frac{1}{2\alpha}}}
\newcommand{\GS}{^{\frac12}}
\newcommand{\HFF}{\frac{}{\rho}}
\newcommand{\HSS}{\frac{}{\rho}}
\newcommand{\HFS}{\rho^{\frac12-\frac{1}{2\alpha}}}
\newcommand{\HSF}{\frac{^{\frac32+\frac{1}{2\alpha}}}{\rho}}
\newcommand{\AF}{\rho}
\newcommand{\AR}{\rho{}^{\frac{1}{2}+\frac{1}{2\alpha}}}
\newcommand{\PF}{\alpha\frac{}{|x|}}
\newcommand{\PS}{\alpha\frac{}{\rho}}
\newcommand{\ds}{\displaystyle}
\newcommand{\Zt}{{\mathcal Z}^{t}}
\newcommand{\er}{e^{-\ell r^{-\ve}}}
\newcommand{\era}{e^{r^{-2\ve}}}
\newcommand{\XPSI}{2\alpha \begin{pmatrix} \frac{x}{|x|^2}\\ 0 \end{pmatrix} - 2\alpha\frac{{}^2}{\rho^2}\begin{pmatrix} x \\ (\alpha +1)|x|^{-\alpha}y \end{pmatrix}}
\newcommand{\Z}{ \begin{pmatrix} x \\ (\alpha + 1)|x|^{-\alpha}y \end{pmatrix} }
\newcommand{\ZZ}{ \begin{pmatrix} xx^{t} & (\alpha + 1)|x|^{-\alpha}x y^{t}\\
     (\alpha + 1)|x|^{-\alpha}x^{t} y &   (\alpha + 1)^2  |x|^{-2\alpha}yy^{t}\end{pmatrix}}
\newcommand{\norm}[1]{\lVert#1 \rVert}
\newcommand{\ve}{\varepsilon}
\newcommand{\ine}{\int_{ \mathcal A}}
\newcommand{\ineb}{\int_{\partial  \mathcal A}}
\newcommand{\lt}{(2\ell^2 \ve  -\ell \ve(n-2-\ve))}
\newcommand{\lto}{(2\ell_0^2 \ve  -\ell_0 \ve(n-2-\ve))}
\newcommand{\sa}{\langle}
\newcommand{\da}{\rangle}

\title[]{Absence of  $L^p$ spectrum for asymptotically flat diffusions in region with cavities}

\begin{abstract}
We study solutions to variable-coefficient elliptic equations of the form $-\D(A(x) \nabla u) = \kappa u$, $\kappa>0$, in an exterior domain $\Om\subset \Rn$, where $A(x)$ is uniformly elliptic and asymptotically flat. Extending Rellich's classical $L^2$ result for the Laplacian, we show that if $u\in L^p(\Om)$ for some $0<p<\frac{2n}{n-1}$, then $u\equiv 0$. The proof uses new monotonicity formulas based on weighted energies and vector fields adapted to the geometry of $A(x)$. Our results highlight a sharper integrability threshold in the variable-coefficient setting.
\end{abstract}
\author{Agnid Banerjee}
\address{School of Mathematical and Statistical Sciences\\ Arizona State University}\email[Agnid Banerjee]{Agnid.Banerjee@asu.edu}

\author{Nicola Garofalo}
\address{School of Mathematical and Statistical Sciences\\ Arizona State University}\email[Nicola Garofalo]{Nicola.Garofalo@asu.edu}

%
%
%


\maketitle

\tableofcontents

\section{Introduction}\label{e:intro}

Understanding the spectral properties of elliptic equations in unbounded domains has long been a central concern in analysis and mathematical physics. A foundational result in this direction was established by F. Rellich in his seminal work \cite{Rel}, where he proved that any nontrivial solution $u$ of the Helmholtz equation 
\begin{equation}\label{helmo}
- \Delta u =  \kappa u
\end{equation}
in the exterior of a cavity $\Om = \{x\in \Rn\mid |x|>1\}$, must fail to belong to $L^2(\Om)$. This result plays a critical role in the spectral theory of the time-independent Schr\"odinger operator $H = - \Delta + V(x)$, where it helps to exclude the existence of positive eigenvalues - an essential ingredient in the formulation of scattering theory and the spectral characterization of quantum systems. One should see in this direction the foundational works by Kato \cite{Ka}, Agmon \cite{A} and Simon \cite{RS}. 

Rellich's theorem not only establishes nonexistence in an integral sense, but also provides sharp pointwise decay conditions under which vanishing must occur. Specifically, if a solution $u$ of \eqref{helmo} satisfies
 \begin{equation}\label{repoint}
u(x) =  o(|x|^{-\frac{n-1}2})
\end{equation}
as $|x|\to \infty$, then $u\equiv 0$ in $\Om$. This threshold is sharp in the pointwise category: for every $\kappa >0$ the spherically symmetric solution 
\begin{equation}\label{f}
u(x) = \int_{\mathbb S^{n-1}} e^{i \sqrt \kappa \sa x,\omega\da} d\sigma(\omega) = c(n,\kappa) |x|^{- \frac{n-2}2} J_{\frac{n-2}{2}}(\sqrt \kappa |x|),
\end{equation}
satisfies the asymptotic bound
\begin{equation}\label{O}
u(x) = O(|x|^{-\frac{n-1}2}),
\end{equation}
but is not square-integrable on $\Om$, in accordance with Rellich's result. Furthermore, thanks to \eqref{O} the eigenfunctions \eqref{f} belong to $L^p(\Om)$ if and only if $p>\frac{2n}{n-1}$, a fact that sharply characterizes the critical integrability threshold for nontriviality.

In our recent work \cite{BGmat}  
we refined Rellich's theorem by establishing that even membership in the endpoint space $L^p(\Om)$ for 
$0< p \le \frac{2n}{n-1}$ suffices to guarantee that $u\equiv 0$. This sharp $L^p$ version of the classical result naturally leads to the question of whether analogous phenomena hold in the more general and physically relevant setting of variable-coefficient operators.

The present paper addresses this question for second-order divergence-form operators with uniformly elliptic and asymptotically flat matrix of the coefficients. Specifically, we consider nontrivial solutions $u$ to the equation 
\begin{equation}\label{mn}
-\D(A(x) \nabla u) = u, 
\end{equation}
in the exterior domain $\Om = \{x\in \Rn \mid |x|>1\}$, where $A(x)$ is a symmetric, uniformly elliptic matrix satisfying appropriate decay assumptions at infinity. Under these structural conditions, our main result (Theorem \ref{main1}) demonstrates that the threshold integrability condition for nontrivial solutions is strictly stronger than in the constant-coefficient case: if $u\in L^{p}(\Om)$ for some  $0<p <\frac{2n}{n-1}$, then $u \equiv 0$ in $\Om$. Notably, the endpoint case $p =\frac{2n}{n-1}$ remains unresolved in the variable-coefficient setting, leaving open the intriguing question of whether this is an artifact of the method or a genuine spectral feature of such operators.

The proof of Theorem \ref{main1} is ultimately reduced to a quantitative lower bound (Theorem \ref{main}) on the spatial growth of $u$, in the spirit of three-sphere inequalities and monotonicity formulae. Our approach is inspired by the techniques introduced in \cite{GSh} and extended in \cite{BGjam} to subelliptic settings, but it also introduces several novel ingredients to handle the intricacies of variable coefficients. Most notably, to capture the asymptotic behavior of eigenfunctions and derive monotonicity, we introduce weighted energies involving the spherically symmetric weight function
\begin{equation}\label{h}
h(|x|) = |x|^\ell e^{\f |x|^{-\ve}},\ \ \ \ \ \ \ell>0,\ \ \ve>0,
\end{equation}
and a suitably chosen vector field
\begin{equation}\label{X}
X(x)= |x|^{s} e^{|x|^{-2\ve}} Z(x),
\end{equation}
carefully adapted to the geometry induced by
$A(x)$, which allow us to derive two crucial monotonicity formulas (Theorems \ref{go1} and \ref{mong2}). These monotonicity results form the analytic core of the paper and may be viewed as a counterpart at infinity to the local almost monotonicity formulae developed in \cite{GL, GLcpam}.

In contrast to the constant coefficient case, the exponential factor in the weight  \eqref{h} plays a pivotal role in the proof of the monotonicity result in Theorem~\ref{go1}. This factor is specifically designed to absorb certain zero-order error terms that emerge in the analysis of the variable coefficient setting and which present a significant analytic obstruction absent in the constant coefficient case. 

Concerning the vector field $X$ in \eqref{X}, it contains two critical features: the vector field $Z$ and the exponential factor, each playing a different role. To illustrate them,  consider the conformal coefficient 
\[
\mu(x) = \sa A(x)\nabla r(x),\nabla r(x)\da,
\]
where we have let $r(x) = |x|$.
We then define  
\[
Z(x) = r(x) \frac{A(x) \nabla r(x)}{\mu(x)}.
\]
This vector field has several remarkable properties, which we collect in Section \ref{np}, particularly in Lemma \ref{dc1}. One should think of $Z$ as a ad hoc variable-coefficient version of the infinitesimal generator of dilations, the Euler vector field $Z(x) = x$, to which it reduces when $A(x) \equiv I_n$. 
On the other hand, the exponential factor in \eqref{X} is carefully tailored to absorb certain first-order error terms that arise in the applications of the Rellich-Payne-Weinberger identity (see Proposition \ref{P:PW}). While this refinement is not needed in the constant coefficient case, it is essential to the success of our arguments in the variable coefficient setting. Ultimately, the proof of the monotonicity Theorems~\ref{go1} and \ref{mong2} hinges on a combination of analytic and geometric ideas which represent one of the main novelties of the present work.

The specific hypothesis on the matrix $A(x) = [a_{ij}(x)]$ are as follows. Besides symmetry, we assume the following uniform ellipticity assumption for some $\lambda\in (0,1)$
\begin{equation}\label{ell}
\lambda |\xi|^2 \leq \langle A(x) \xi, \xi \rangle  \leq \lambda^{-1} |\xi|^2,
\end{equation}
for every $x\in \Om$ and every $\xi\in \Rn$.
Furthermore, we require that there exist $\alpha, C>0$ such that, with
\begin{equation}\label{B}
B(x) =[b_{ij}(x)] = A(x) - I_n,\ \ \ \ \ x\in \Om,
\end{equation}
we have the following decay at infinity 
\begin{equation}\label{st}
|x| |\nabla B(x)| + |B(x)| \leq \frac{C}{|x|^{\alpha}},\ \ \ \ \ \ \ \ x\in \Om.
\end{equation}
The hypothesis \eqref{ell} and \eqref{st} will be assumed throughout the rest of the paper without further reference to them. The main result in this paper is the following.

\begin{thrm}\label{main1}
Let $u$ be a solution of \eqref{mn} in $\Om$. If $u\in L^{p}(\Om)$ for some  $0<p <\frac{2n}{n-1}$, then $u \equiv 0$ in $\Om$.
\end{thrm}
As noted above, in contrast to the constant coefficient case, Theorem \ref{main1} does not include the endpoint $p = \frac{2n}{n-1}$. Theorem \ref{main1} will be obtained as a consequence of the following result.

\begin{thrm}\label{main}
Let $u\in W^{1,2}_{loc}(\Om)$ be a solution to \eqref{mn} such that $u\not\equiv 0$. For every $\delta \in (0,1)$, there exist  $\bar R>>1$ (depending on $u, n, \la,\alpha,C, \delta$), and $C_1>0$ (depending on $n,\la,\alpha,C,\delta$),
 such that for all $R\geq \bar R$
\begin{equation}\label{goal}
\int_{ B_{2R} \setminus B_R} u^2 dx \geq C_1 R^{1-\delta}.
\end{equation}
\end{thrm}
We emphasize that the lower bound on annular integrals \eqref{goal} reflects non-integrability at infinity. The structure of the paper is as follows. In Section \ref{S:key} we introduce the basic functionals $F(\ell,\rho)$ and $G(\rho)$ (see Definitions \ref{D:F} and \ref{D:G}), and state, without proof, their key monotonicity properties. For this, see Theorems \ref{go1} and \ref{mong2}, respectively. The final objective of the section is to derive Theorems \ref{main} and \ref{main1} from these results. Using Theorem \ref{go1}, we establish the important Lemma \ref{post}. We also prove a crucial conditional positivity result, see Lemma \ref{post1}. Section \ref{pf} is devoted to the proof of Theorems \ref{go1} and \ref{mong2}. As we have said, this section forms the analytic core of the paper, and unsurprisingly it represents the most technical part of the latter. The final Section \ref{np} is an appendix in which we introduce the relevant notation and collect some basic auxiliary results that are extensively used in the main part of the paper.

\vskip 0.2in

\section{Two key motononicity results}\label{S:key}

In this section we state without proofs the two main technical results of the present work, Theorems \ref{go1} and \ref{mong2}, and show how to derive from them Theorems \ref{main} and \ref{main1}. As we have mentioned, the actual proofs of these results are deferred to Section \ref{pf}.
Let  $u$ be a given solution to \eqref{mn}. With $h(r)$ as in \eqref{h}, we consider the function
\begin{equation}\label{deZv}
v= h(r) u = r(x)^\ell e^{\f r(x)^{-\ve}} u.
\end{equation}
The reader should bear in mind that such function depends on the parameters $\ell$ and $\ve$, and therefore we should actually indicate it with $v_{\ell,\ve}$, but this would overburden the notation. When we will make a specific choice of $\ell$ and $\ve$, the function $v$ will be adjusted to reflect such choices.

Consider now the vector field $Z$ in \eqref{F}. We preliminarily observe that by applying Lemma \ref{L:rad}, we obtain
\begin{equation}\label{hg}
\left(\frac{Z h}{r} \right)^2 \mu  = \langle A\nabla  h, \nabla h \rangle,
\end{equation}
and
\begin{equation}\label{hg1}
\bigg(\left(\frac{Zv}{r}\right)^2 \mu -\langle A \nabla v, \nabla v\rangle \bigg)=  r^{2\ell} e^{2\ell r^{-\ve}}\bigg(\left(\frac{Zu}{r}\right)^2 \mu -\langle A \nabla u, \nabla u\rangle \bigg),
\end{equation}
where $\mu$ is the conformal factor in \eqref{mu}.
The functional $F(\ell, \rho)$ in the following definition will play a critical role in the results in this paper. 
 
\begin{dfn}\label{D:F}
For every $\rho>1$, $\ell>0$, we let  
\begin{align}\label{g1}
&F(\ell, \rho) = 2 \int_{\partial B_\rho} \era \left(\frac{Zv}{r} \right)^2 \mu  - \int_{\partial B_\rho} \era \langle A\nabla v, \nabla v \rangle +\int_{\partial B_\rho} \era v^2 \\
& + \f(\f -n+2) \int_{\partial B_\rho} r^{-2} \era v^2 \mu  - \lt \int_{\partial B_\rho} r^{-2-\ve} \era  v^2    \mu
\notag\\
& + \f^2 \ve^2 \int_{\partial B_\rho} r^{-2-2\ve} \era  v^2\mu   - \int_{\partial B_\rho} r^{-1} \era v^2 \mu .
\notag
\end{align}
\end{dfn}

\begin{rmrk}\label{R:v}
Since, as we have noted above, the function $v$ in \eqref{g1} depends on the solution $u$ to \eqref{mn}, as well as on the parameters $\ell, \ve$, a more accurate notation for the functional in Definition \ref{D:F} should have been $F(u,\ell,\ve,\rho)$. The reader should keep this in mind when, in the results that follow, we restrict the range of the parameters $\ell, \ve, \rho$. In particular, when in Lemma \ref{post} we write $F(\ell_0,\rho)$, it means that the function $v$ involved is the one corresponding to the choice $\ell = \ell_0$ in \eqref{deZv}, i.e.,
\[
v=  r(x)^{\ell_0} e^{\ell_0 r(x)^{-\ve}} u.
\]
\end{rmrk}

The first main result of this section is the following theorem. Its proof is deferred to Section \ref{pf}.

\begin{thrm}\label{go1}
Let $0<\ve <\frac{\alpha}{2}$. There exist $r_0(n,\alpha,\ve), \tilde \ell(n,\alpha,\ve)>>1$, such that for all $\ell \geq \tilde \ell$ and $\rho \geq  r_0$, the function
$$\rho\ \longrightarrow\  \frac{F(\ell, \rho)}{\rho^{n-1}}$$ is  non-decreasing.
\end{thrm}

The next positivity result is a crucial consequence of Theorem \ref{go1}.

\begin{lemma}\label{post}
Assume $0<\ve <\frac{\alpha}{2}$, and that  $u \not \equiv 0$ in $\Om$. There exist $R_0 = R_0(u,n,\alpha), \ell_0(u,n,\alpha,\ve)>>1$, such that  for all $\rho \geq R_0$ 
\[
F(\ell_0, \rho)>0.
\]
\end{lemma}

\begin{proof}
We first note that, for $\rho$ large enough, we have  
\begin{align}\label{hg2}
& \f(\f -n+2) \int_{\partial B_\rho} r^{-2}\era v^2 \mu  - \lt \int_{\partial B_\rho} r^{-2-\ve} \era  v^2   \mu
\\
& + \f^2 \ve^2 \int_{\partial B_\rho} r^{-2-2\ve} \era v^2\mu\ 
 \geq\ \frac{\f(\f -n+2)}{2} \int_{\partial B_\rho} r^{-2}\era v^2\mu.
\notag
\end{align}
Using \eqref{hg}-\eqref{hg2}, it thus follows that for all $\rho$ sufficiently large
\begin{align}\label{gh1}
F(\ell, \rho) & \geq \rho^{2\ell}e^{2\ell \rho^{-\ve}} \int_{\partial B_{\rho}}\bigg(\left(\frac{Zu}{r}\right)^2 \mu -\langle A \nabla u, \nabla u\rangle \bigg) \era
\\
& +\rho^{2\ell}e^{2\ell \rho^{-\ve}} \int_{\partial B_\rho} \bigg(\frac{\f(\f -n+2)}{2}  r^{-2} u^2  - r^{-1} u^2 \bigg)\mu\era.\notag
\end{align}
Since $u \not \equiv 0$ and solves  \eqref{mn} in $\Rn \setminus B_1$, by the unique continuation property, see \cite{GL, GLcpam}, there exists $R_0 = R_0(u,n,\alpha)>>1$ such that 
\[
\int_{\partial B_{R_0}} u^2>0.
\]
Using this bound in \eqref{gh1}, we infer the existence of $\ell_0(u,n,\alpha,\ve)>>1$ such that $F(\ell_0, R_0) >0$. From this positivity, the desired conclusion follows from Theorem \ref{go1}.

\end{proof}

We next introduce a second functional which is also critical to our results. 

\begin{dfn}\label{D:G}
With $u$ as in \eqref{mn}, we let $w=r^{\frac{n-1}{2}}  u$.
For every $\rho>1$, we define
\begin{align}\label{g2}
&G(\rho)\overset{def}= 2 \int_{\partial B_\rho} \left(\frac{Zw}{r} \right)^2 \mu  - \int_{\partial B_\rho} \langle A\nabla w, \nabla w \rangle 
 + \int_{\partial B_\rho}  w^2  -\frac{(n-1)(n-3)}{4} \int_{\partial B_\rho} \frac{w^2}{r^{2}}\mu.
\end{align}
\end{dfn}

We will need the following delicate conditional positivity  result for $G(\rho)$. 

\begin{lemma}\label{post1}
Let $u$ be as in \eqref{mn}. Suppose that for some  $\delta\in (0,1)$ we have
\begin{equation}\label{hyp}
\liminf_{R \to \infty}\frac{1}{R^{1-\delta}}\int_{R<|x| < 2R}  u^2=0.
\end{equation}
Then there exists $r_j \to \infty$ such that $G(r_j)>0$. 
\end{lemma}
             
\begin{proof}
\underline{Step 1:} We show that \eqref{hyp} implies the existence of a sequence $r_j \to \infty$ such that
\begin{equation}\label{hyp1}
\int_{\partial B_{r_j}}  Z(w^2)\ \mu \leq 0.   
\end{equation}
To prove \eqref{hyp1} we argue by contradiction and assume that there exist $R_1>R_0$ such that for every $\rho \geq R_1$
\begin{equation}\label{hyp2}
\int_{\partial B_\rho}  Z(w^2)\ \mu >0.
\end{equation}  
We will prove that \eqref{hyp2} implies 
\begin{equation}\label{hyp12}
\liminf_{R \to \infty}\frac{1}{R^{1-\delta}}\int_{R<|x| < 2R}  u^2>0,
\end{equation}
thus contradicting \eqref{hyp}.
Keeping in mind that $w=r^{\frac{n-1}{2}} u$, we apply the divergence theorem to the domain $\mathcal A=\{x\in \Rn\mid R_1< |x| <R\}$, for $R\geq 2R_1$. Using (iv), (v) in Lemma \ref{dc1} and \eqref{Zrad}, we obtain the following crucial strict positivity
\begin{align}\label{t10}
& \int_{\partial  \mathcal A}  r^{-n+\delta}  w^2  \langle Z, \nu \rangle \mu = \int_{\mathcal A} \D( r^{-n+\delta}  w^2 \mu Z)
\\
& = \int_{\mathcal A} r^{-n+\delta}  w^2 (n+O(r^{-\alpha}))\mu + \int_{\mathcal A} (\delta-n)r^{-n+\delta}  w^2 \mu  + \int_{\mathcal A} r^{-n+\delta}  Z(w^2)\mu
\notag
\\
& = \delta \int_{\mathcal A} r^{-n+\delta} w^2 \mu + \int_{\mathcal A} O(1) r^{-n+\delta -\alpha}   w^2 \mu + \int_{\mathcal A} r^{-n+\delta}  Z(w^2)\mu\ \ge\ C(u)\ >\ 0.
\notag                       
\end{align} 
To achieve the last inequality in \eqref{t10} observe that, since $\delta>0$, if we choose  $R_1$ large enough, one has
\[
\delta \int_{\mathcal A} r^{-n+\delta} w^2 \mu + \int_{\mathcal A} O(1) r^{-n+\delta -\alpha}   w^2 \mu \geq 0,
\]
whereas \eqref{hyp2} gives              
\[
\int_{\mathcal A} r^{-n+\delta}  Z(w^2)\mu = \int_{R_1}^{R} r^{-n+\delta} \left(\int_{\partial B_r}  Z(w^2)\mu\right)   dr \ge   \int_{R_1}^{2R_1} r^{-n+\delta} \left(\int_{\partial B_r}  Z(w^2) \mu\right) dr= C(u)>0.
\]   
By (i) in Lemma \ref{dc1}, and by the fact that $u = r^{\frac{1-n}2} w$, we have
\begin{align*}
R^\delta \int_{\p B_R}    u^2 \mu & \ge R^\delta \int_{\p B_R}  u^2 \mu - R_1^\delta \int_{\p B_{R_1}}    u^2 \mu
 = \int_{\partial \mathcal A}  r^{-n+\delta}  w^2  \langle Z, \nu \rangle \mu 
 \\
 & =  \int_{\partial \mathcal A}  r^{-n+\delta}  w^2 \langle Z, \nu \rangle \mu \ \ge\ C(u)\ >\ 0,
\end{align*}
where in the last inequality we have used \eqref{t10}. We thus obtain for all $R\geq 2R_1$ 
\begin{equation}\label{t11}
\int_{R<|x|<2R} u^2 \mu  \geq  C(u) R^{1-\delta}.
\end{equation}
Since by \eqref{ellmu} one has $\mu \approx 1$, \eqref{t11} implies that \eqref{hyp12} must hold, which, as we have said, contradicts \eqref{hyp}. Therefore \eqref{hyp1} holds for a sequence $r_j \to \infty$. 

\vskip 0.2in

\noindent \underline{Step 2:} We next show that, on this sequence, we must have $G(r_j)>0$ for every $j\in \mathbb N$.
Let $\ell_0$ be as in Lemma \ref{post}, so that 
\[
F(\f_0, R_0) >0.
\] 
Keeping Remark \ref{R:v} in mind, from \eqref{g1} we have for $\rho \geq R_0$
\begin{align*}
&0< F(\ell_0, \rho) = 2 \int_{\partial B_\rho} \era \left(\frac{Zv}{r} \right)^2 \mu  - \int_{\partial B_\rho} \era \langle A\nabla v, \nabla v \rangle +\int_{\partial B_\rho} \era v^2 \\
& + \f_0(\f_0 -n+2) \int_{\partial B_\rho} r^{-2} \era v^2 \mu  - \lto \int_{\partial B_\rho} r^{-2-\ve} \era  v^2    \mu
\notag\\
& + \f_0^2 \ve^2 \int_{\partial B_\rho} r^{-2-2\ve} \era  v^2 \mu - \int_{\partial B_\rho} r^{-1} \era v^2 \mu .
\notag
\end{align*}
From this inequality, proceeding similarly to \eqref{hg1}, we obtain for $\rho \geq R_0$
\begin{align}\label{t1}
& 0< F(\f_0, \rho) \leq \int_{\partial B_{\rho}} \left(\frac{Zv}{\rho}\right)^2e^{\rho^{-2\ve}}\mu + \rho^{2\f_0 - n+1}e^{2\ell_0 \rho^{-\ve}}e^{\rho^{-2\ve}}\int_{\partial B_{\rho}}\bigg(\left(\frac{Zw}{\rho}\right)^2 \mu -\langle A \nabla w, \nabla w\rangle \bigg)\\          
          & +   \rho^{2\f_0 - n+1}e^{2\ell_0 \rho^{-\ve}}e^{\rho^{-2\ve}} \int_{\partial B_{\rho}} w^2 \notag\\
          &+ \rho^{2\f_0} e^{2\ell_0 \rho^{-\ve}}e^{\rho^{-2\ve}}   \bigg( -\frac{1}{\rho} +\frac{\f_0(\f_0- n+2)}{\rho^2} + \frac{\f_0^2 \ve^2}{\rho^{2+2\ve}} \bigg) \int_{\partial B_{\rho}} u^2 \mu.\notag               
\end{align}
We now let $$\ell_1= \f_0 -\frac{n-1}{2}.$$
Recalling that $u=r^{-\frac{n-1}{2}} w$, and that $v= r(x)^{\ell_0} e^{\f r(x)^{-\ve}} u$ (see Remark \ref{R:v}), we can express $v$ in terms of $w$ as follows 
\[
v = r(x)^{\ell_0 - \frac{n-1}{2}} e^{\ell_0 r(x)^{-\ve}} w = r(x)^{\ell_1} e^{\ell_0 r(x)^{-\ve}} w.
\]
By (i) in Lemma \ref{dc1} we thus find
\[
Zv = \left(\ell_1 r^{\ell_1}  - \ve \ell_0 r^{\ell_1 - \ve}\right) e^{\ell_0 r(x)^{-\ve}} w + r(x)^{\ell_1} e^{\ell_0 r(x)^{-\ve}} Zw.
\]
Using this information, and expanding the first term in the right-hand side of \eqref{t1}, we obtain the following 
\begin{align}\label{t2}
& \int_{\partial B_{\rho}} \left(\frac{Zv}{r}\right)^2e^{r^{-2\ve}}\mu= \left(\f_1^2 \rho^{2(\f_1-1)} + \f_0^2 \ve^2 \rho^{2(\f_1-1-\ve)} - 2\f_1\f_0\ve \rho^{(2(\f_1-1)-\ve)}\right)e^{2\ell_0 \rho^{-\ve}}e^{\rho^{-2\ve}} \int_{\partial B_\rho} w^2\mu\\
            &+ \rho^{2\f_1} e^{2\ell_0 \rho^{-\ve}}e^{\rho^{-2\ve}}  \int_{\partial B_\rho}\left(\frac{Zw}{r}\right)^2 \mu \notag   \\
            & + (\f_1\rho^{2(\f_1 -1)} - \f_0 \ve \rho^{2\f_1-2-\ve}) e^{2\ell_0 \rho^{-\ve}}e^{\rho^{-2\ve}} \int_{\partial B_{\rho}} Z(w^2)\mu.\notag
\end{align}   
Inserting \eqref{t2} in \eqref{t1}, and keeping Definition \ref{D:G} in mind, we find
\begin{align}\label{t4}
&  0< F(\f_0, \rho) \leq \rho^{2\f_1}   e^{2\ell_0 \rho^{-\ve}}e^{\rho^{-2\ve}} G(\rho)\\
                                 & + \rho^{2\f_0} e^{2\ell_0 \rho^{-\ve}}e^{\rho^{-2\ve}}\bigg(-\frac{1}{\rho}     +\frac{\f_0(\f_0- n+2)}{\rho^2} + \frac{\f_0^2 \ve^2}{\rho^{2+2\ve}}+ \frac{(n-1)(n-3)}{4 \rho^2}  \bigg) \int_{\partial B_\rho} u^2\mu\notag\\
                                &  +\rho^{2\f_0} e^{2\ell_0 \rho^{-\ve}}e^{\rho^{-2\ve}}\bigg(    \frac{\f_1^2}{\rho^2}  + \frac{\f_0^2 \ve^2}{\rho^{2+2\ve}}    \bigg)\int_{\partial B_\rho} u^2\mu\notag\\
                                & + (\f_1\rho^{2(\f_1 -1)} - \f_0 \ve \rho^{2\f_1-2-\ve}) e^{2\ell_0 \rho^{-\ve}}e^{\rho^{-2\ve}} \int_{\partial B_{\rho}} Z(w^2)\mu.\notag                                 
\end{align}
Next, we observe that, for $\rho$ large enough, we can achieve
\begin{align}\label{sh}
& \rho^{2\f_0} e^{2\ell_0 \rho^{-\ve}}e^{\rho^{-2\ve}}\bigg(-\frac{1}{\rho}     +\frac{\f_0(\f_0- n+2)}{\rho^2} + \frac{2\f_0^2 \ve^2}{\rho^{2+2\ve}}+ \frac{(n-1)(n-3)}{4 \rho^2}  \bigg) \int_{\partial B_\rho} u^2\mu
\\
&  +\rho^{2\f_0} e^{2\ell_0 \rho^{-\ve}}e^{\rho^{-2\ve}}\bigg(\frac{\f_1^2}{\rho^2}  + \frac{\f_0^2 \ve^2}{\rho^{2+2\ve}}\bigg)\int_{\partial B_\rho} u^2\mu  \leq 0.
\notag
\end{align}  
In fact, after elementary simplifications, this inequality is equivalent to having 
\begin{align*}
&  \left(\f_1^2  + \frac{\f_0^2 \ve^2}{\rho^{2\ve}}\right)\int_{\partial B_\rho} u^2\mu 
\\
& \le \left(\rho - \f_0(\f_0- n+2) - \frac{2\f_0^2 \ve^2}{\rho^{2\ve}} - \frac{(n-1)(n-3)}{4}\right) \int_{\partial B_\rho} u^2\mu,
\end{align*}   
which, since Lemma \ref{dc1} implies that $\mu(x) \to 1$ uniformly as $r=|x| \to \infty$, is obviously valid for large $\rho$. The estimates \eqref{t4}, \eqref{sh} thus imply
\begin{align}\label{t8}
& 0<  F(\f_0, \rho) \leq \rho^{2\f_1}   e^{2\ell_0 \rho^{-\ve}}e^{\rho^{-2\ve}} \left\{G(\rho) + \rho^{-2}(\f_1 - \f_0 \ve \rho^{-\ve})  \int_{\partial B_{\rho}} Z(w^2)\mu\right\}.                                 
\end{align}          
Since for large enough $\rho$, depending on $\ell_0$, we can ensure that
\[
(\f_1 - \f_0 \ve \rho^{-\ve}) > 0,
\]
it is now clear that the combination of \eqref{t8} and \eqref{hyp1} implies the conclusion of the lemma.                                                                                                          
                                          
\end{proof}

The following monotonicity theorem is the second main result in this section.

\begin{thrm}\label{mong2}
Given $\delta \in (0,1)$,  there exists $R_1 = R_1(u,n,\alpha,\delta)>>1$, such that the function 
\[
\rho\ \longrightarrow\  \frac{G(\rho)}{\rho^{n-1-\delta}}
\]
is non-decreasing on the interval $[R_1,\infty)$.
\end{thrm}


\begin{rmrk}
We recall that in Lemma 4.1 of \cite{BGjam}, a stronger monotonicity formula was established for the Baouendi-Grushin operator
\begin{equation}\label{pbeta}
\mathscr B_{\alpha} = \Delta_z  + \frac{|z|^{2\alpha}}4 \Delta_\sigma,
\end{equation}
where $z\in \Rm$, $\sigma\in \R^k$, and $\alpha>0$.  In comparison, the result in that paper corresponds to formally letting $\delta = 0$ in the statement of Theorem \ref{mong2}, as well as in assumption \eqref{hyp} in Lemma \ref{post1}.
In the variable-coefficient setting considered in the present work, the parameter $\delta$ becomes essential to absorb various error terms in a controlled manner, and its strict positivity plays a crucial role in establishing the key inequality \eqref{t10}.
\end{rmrk}

Theorem \ref{mong2} combined with Lemma \ref{post1} implies the following.

\begin{lemma}\label{post2}
Under the hypothesis of Lemma \ref{post1}, there exist $C_2>0$ and $R_2$ large enough such that for all $\rho\geq R_2$
\[ 
\frac{G(\rho)}{\rho^{n-1-\delta}}
 \geq C_2.
\]
\end{lemma}

We are finally ready to provide the                                             
                                             
\begin{proof}[Proof of Theorem \ref{main}]
We assume that $u \not \equiv 0$ in $\Om$, therefore by unique continuation, for any $R$ we must have $u\not\equiv 0$ in the set $\{|x|>R\}$ .
It suffices to show that \eqref{hyp2} hold, since the latter implies \eqref{hyp12}, as we have seen in the proof of Lemma \ref{post1}.
First, we note that from the definition \eqref{g2} of $G$, it follows that
\begin{align*} 
G(\rho)  & = 
2 \int_{\partial B_\rho} \left(\frac{Zw}{r} \right)^2 \mu + 2 \int_{\partial B_\rho} w^2  - \int_{\partial B_\rho} \langle A\nabla w, \nabla w \rangle 
- \int_{\partial B_\rho}  w^2\left(1  +\frac{(n-1)(n-3)\mu}{4 r^2}\right) 
\\
& \leq 2 \int_{\partial B_{\rho}} \left(\frac{Zw}{r}\right)^2 \mu + 2 \int_{\partial B_\rho} w^2,
\end{align*}
 for all $\rho$ large enough.
This inequality, combined with Lemma \ref{post2}, implies
\begin{align}\label{zn1}
\int_{\partial B_{\rho}} \left(\frac{Zw}{r}\right)^2 \mu +  \int_{\partial B_\rho} w^2 \geq C \rho^{n-1-\delta},
\end{align}
for all $\rho$ sufficiently large. Recalling that $u = r^{-\frac{n-1}{2}} w$, and using that $Z(r^{\frac{n-1}2}) = - \frac{n-1}2 r^{\frac{n-1}2}$ (see \eqref{Zrad}), we find
\begin{align*}
&  \left(\frac{Zu}{r}\right)^2= r^{1-n} \left( \left(\frac{n-1}{2}\right)^2 \frac{w^2}{r^2}   +\left(\frac{Zw}{r}\right)^2 - \frac{n-1}{r} w \frac{Zw}{r} \right).
\end{align*}
We thus have 
\begin{align*}
&\int_{\partial B_\rho} \left(\frac{Zu}{r}\right)^2 + \int_{\partial B_\rho} u^2  \geq \rho^{1-n} \bigg\{ \int_{\partial B_\rho} \left(\frac{Zw}{r}\right)^2 + \int_{\partial B_\rho} w^2
+ \left(\frac{n-1}{2\rho}\right)^2 \int_{\partial B_\rho}    w^2  -   \frac{n-1}{\rho} \int_{\partial B_\rho} \left|\frac{Zw}{r}\right||w| \bigg\}
\\
& \geq \frac{1}{2} \rho^{1-n} \left(\int_{\partial B_\rho} \left(\frac{Zw}{r}\right)^2 + \int_{\partial B_\rho} w^2\right), 
\notag
\end{align*}
for large enough $\rho$. Combined with \eqref{zn1}, this bound implies the following inequality 
\begin{equation}\label{c4}
\int_{\partial B_{\rho}} \left(\frac{Zu}{\rho}\right)^2 \mu + \int_{\partial B_{\rho}} u^2 \geq C_4\rho^{-\delta},
\end{equation}
for large $\rho$. Noting that \eqref{ell} and \eqref{rad} give
\begin{align*}
& \int_{\frac{5R}{4}<|x|<\frac{7R}{4}} |\nabla u|^2  +    \int_{\frac{5R}{4}<|x|<\frac{7R}{4}} u^2 \ge \la \int_{\frac{5R}{4}<|x|<\frac{7R}{4}} \sa A(x)\nabla u,\nabla u\da  +    \int_{\frac{5R}{4}<|x|<\frac{7R}{4}} u^2
\\
& = \int_{\frac{5R}{4}}^{\frac{7R}{4}} \int_{S_\rho} \left[\sa A(x)\nabla u,\nabla u\da +  u^2\right] \ge \int_{\frac{5R}{4}}^{\frac{7R}{4}} \int_{S_\rho}  \left[\left(\frac{Zu}{\rho}\right)^2 \mu +  u^2\right], 
\end{align*}
from \eqref{c4} we conclude 
\begin{equation}\label{c5}
\int_{\frac{5R}{4}<|x|<\frac{7R}{4}} |\nabla u|^2  +    \int_{\frac{5R}{4}<|x|<\frac{7R}{4}} u^2 \geq C R^{1-\delta},
\end{equation}
for all large $R$. Using the Caccioppoli inequality  
\[
\int_{\frac{5R}{4}<|x|<\frac{7R}{4}} |\nabla u|^2      \leq C \int_{R<|x|<2R} u^2
\]
in \eqref{c5}, we finally obtain for a new $C_1>0$ and all large $R$,
\[
\int_{R<|x|<2R} u^2       > C_1 R^{1-\delta}.
\]
This proves \eqref{hyp2}, thus completing the proof of the theorem.
                                                                                                            
\end{proof}

With Theorem \ref{main} in hands, we turn to the

\begin{proof}[Proof of Theorem \ref{main1}]
                                             
Let $u$ solve \eqref{mn}, and assume that $u \in L^p(\Rn \setminus B_1)$. We argue by contradiction and suppose that $u\not\equiv 0$. We first look at the case $p \in (0, 2)$. By the unique continuation property of the equation \eqref{mn}, we must have $u\not\equiv 0$ in $\Rn \setminus B_2$, and therefore $||u||_{L^\infty(\Rn \setminus B_2)}>0$.
On the other hand, for any $p>0$ one has the following estimate due to Moser
\begin{equation*}
||u||_{L^{\infty}(\Rn \setminus B_2)} \leq  \left(\int_{\Rn \setminus B_1} |u|^p\right)^{1/p},
\end{equation*}
for some $C = C(n, \lambda, p)>0$,           
see for instance \cite[Theor. 4.1]{HL}. Therefore, we also have  $||u||_{L^\infty(\Rn \setminus B_2)} < \infty$. If we now take $\delta=1/2$ in Theorem \ref{main}, then for all large $R$ we have
\begin{equation*}
C_1 R^{1/2} \le \int_{ B_{2R} \setminus B_R} u^2 dx \le ||u||^{2-p}_{L^\infty(\Rn \setminus B_2)} \int_{R<|x|<2R} |u|^p dx.
\end{equation*}
Since
\[
u \in L^p(\Rn \setminus B_1) \ \Longrightarrow\ \underset{R\to \infty}{\lim}\ \int_{R<|x|<2R} |u|^p dx = 0,
\]
we have reached a contradiction. This shows that $u\equiv 0$ in $\Rn \setminus B_1$.

We now look at the case $2\le p < \frac{2n}{n-1}$ and note that in such case we have $n\left(1-\frac 2p\right)< 1$.
Therefore, we can write 
\begin{equation}\label{delp}n\left(1-\frac 2p\right)=1-\delta,
\end{equation}
for some $\delta = \delta_p\in (0,1)$. 
Applying H\"older inequality to \eqref{goal} in Theorem \ref{main}, we obtain for all $R\geq R_\delta$
\begin{equation*}
C_1 R^{1-\delta} \le C(n, p) \left(\int_{R<|x|<2R} |u|^p dx\right)^{\frac 2p} R^{n(1-\frac 2p)}.
\end{equation*}
Using \eqref{delp}, we obtain from this inequality    
\[
0<C_1   \le C(n,p) \left(\int_{R<|x|<2R} |u|^p dx\right)^{\frac 2p},
\]
which again implies a contradiction.
This finishes the proof of the theorem.                                    
                                             
\end{proof}

\vskip 0.2in
 
\section{Proof of the monotonicity theorems}\label{pf}

This section is devoted to proving Theorems \ref{go1} and \ref{mong2}. We begin with the 
 
\begin{proof}[Proof of Theorem \ref{go1}]
Let $u$ be a solution to \eqref{mn} and denote by $v$ the function defined by \eqref{deZv}. Given $1<t<\tau<\infty$, consider the annular  region 
\begin{equation}\label{om}
\mathcal A =\{x\in \Rn\mid t <|x|<\tau\} \subset \Rn \setminus B_1.
\end{equation}
Our first step is to apply the Rellich-Payne-Weinberger identity in Proposition \ref{P:PW} with $U = \mathcal A$, $f = v$ and the special choice \eqref{X} of the vector field, which for ease of notation we momentarily write $X = \sigma(r) Z$, where $\sigma(r) = r^s e^{r^{-2\ve}}$. The parameter $s$ will be later chosen as in \eqref{choice1} in order to eliminate some terms in \eqref{hj4}. Before applying the identity, we make some observations. On $\p \mathcal A$ we have $\nu = \pm \nabla r$, depending on whether we are on the outer or inner part of the boundary. We thus find
\[
Xv \sa A\nabla v,\nu\da = \sigma(r) Zv \sa A\nabla v,\pm\nabla r\da = \pm \frac{\sigma(r)}{r} (Zv)^2\ \mu = \pm \sigma(r) \left(\frac{Zv}r\right)^2\ r\ \mu.
\]
On the other hand, we also have
\[
\sa Z,\nu\da = \pm \sa Z,\nabla r\da = \pm Zr = \pm r,
\]
where the last equality is justified by (i) in Lemma \ref{dc1}. Combining these identities, we find on $\p \mathcal A$
\begin{equation}\label{1}
Xv \sa A\nabla v,\nu\da = r^{s} \era \left(\frac{Zv}r\right)^2 \sa Z,\nu\da \mu.
\end{equation}
Also notice that by (i) and (v) in Lemma \ref{dc1}, we find
\begin{equation}\label{2}
\D X = \sigma(r) \D Z + \sigma'(r) Zr = (n+ s) r^s \era + r^{s-\alpha} \era O(1) - 2
\ve r^{s-2\ve}\era.
\end{equation}
One more observation is that
\begin{align}\label{3}
-2 a_{ij} [D_i, X] v D_j v & = \sigma(r) a_{ij} [D_i,Z] v D_j v +  \sigma'(r) Zv \sa A\nabla r,\nabla v\da 
\\
& = - 2 \sigma(r) \sa A\nabla v,\nabla v\da - 2 r \sigma'(r) \left(\frac{Zv}{r}\right)^2 \mu - 2 \sigma(r) a_{ij} \left[[D_i,Z] -D_i\right] v  D_j v
\notag\\
& = - 2 r^s \era \sa A\nabla v,\nabla v\da - 2 s r^s \era  \left(\frac{Zv}{r}\right)^2 \mu + 4\ve r^{s-2\ve}\era \left(\frac{Zv}{r}\right)^2 \mu
\notag\\
& \ \ \ + r^{s-\alpha}\era  \sa A\nabla v,\nabla v\da O(1),
\notag
\end{align}
where to pass from the first to the  second equality we have used \eqref{Zuf}, whereas we have used (vi) in Lemma \ref{dc1} to estimate the last term in the third equality.
Applying Proposition \ref{P:PW}, keeping \eqref{1}, \eqref{2} and \eqref{3} in mind, we thus obtain 
\begin{align}\label{k4}
& 2 \ine r^{s} \era Zv \D(A\nabla v)  = 2 \ineb r^{s} \era \left(\frac{Zv}{r}\right)^2 \langle Z, \nu\rangle \mu - \ineb r^{s}\era \langle A \nabla v, \nabla v\rangle \langle Z, \nu \rangle \\
&- 2s \ine r^{s-2} \era (Zv)^2 \mu  +4 \ve \ine r^{s-2-2\ve} \era (Zv)^2 \mu  +(n+s-2) \ine r^{s}\era \langle A\nabla v, \nabla v\rangle 
\notag\\
&- 2\ve \ine r^{s-2\ve} \era \langle A \nabla v, \nabla v \rangle  + \ine r^{s-\alpha} \era  \langle A \nabla v, \nabla v\rangle O(1). 
\notag
\end{align} 
We next express the term on the left-hand side of \eqref{k4} using the  equation \eqref{mn} satisfied by $u$.   
From \eqref{deZv}, with $k(r) = r^{-\f} e^{- \f r^{-\ve}}$, we have
\[
u= k(r) v,
\]
and therefore we obtain from \eqref{Zf}, \eqref{Zuf}
\begin{align*}
\D(A\nabla u) & = k(r) \D(A\nabla v) + v \D(A\nabla k(r)) + 2 \frac{k'(r)}{r}  Z v\ \mu.
\end{align*}
Keeping in mind that $u$ satisfies \eqref{mn}, we infer the following equation satisfied by $v$ 
\begin{align}\label{eqv}
\D(A\nabla v)= -v \left(1+  \frac{\D(A\nabla k(r))}{k(r)}\right) - 2 \frac{k'(r)}{rk(r)}  Z v\ \mu.
\end{align}
Next, using (i), (iii) and (v) in Lemma \ref{dc1}, we find
\begin{align}\label{eqk}
\D(A\nabla k(r)) & = \D(\frac{k'(r)}{r} \mu Z) = \frac{k'(r)}{r} \mu \D Z + Z(\frac{k'(r)}{r} \mu)
\\
& = \frac{k'(r)}{r} \left[n + O(r^{-\alpha})\right] \mu + \frac{k'(r)}{r} O(r^{-\alpha}) + \left(k''(r) - \frac{k'(r)}{r}\right) \mu.
\notag
\end{align}
A computation now gives
\[
\frac{k'(r)}{rk(r)} = \frac{\ell}{r^2} (\ve r^{-\ve}-1),
\]
and
\[
\frac{k''(r)}{k(r)} - \frac{k'(r)}{r k(r)} =  \frac{\ell^2}{r^2} (\ve r^{-\ve}-1)^2 -\ell \left[2 r^{-2}(\ve r^{-\ve}-1) + \ve^2 r^{-2-\ve}\right].
\]
Using the latter two equations in \eqref{eqk}, we find
\begin{align*}
\frac{\D(A\nabla k(r))}{k(r)} & = \frac{\ell}{r^2} (\ve r^{-\ve}-1) \left[(n + O(r^{-\alpha}))\ \mu +  O(r^{-\alpha})\right]
\\
& + \left\{\frac{\ell^2}{r^2} (\ve r^{-\ve}-1)^2 -\ell \left[2 r^{-2}(\ve r^{-\ve}-1) + \ve^2 r^{-2-\ve}\right]\right\} \mu
\notag
\\
& = \left\{\frac{\ell^2}{r^2} (\ve r^{-\ve}-1)^2 + \frac{\ell}{r^2} (\ve r^{-\ve}-1) (n-2+O(r^{-\alpha})) -\ell \ve^2 r^{-2-\ve}\right\} \mu.
\notag 
\end{align*}
Substitution in \eqref{eqv} gives
\begin{align}\label{eqv2}
\D(A\nabla v) & = -v \left[1+  \left(\frac{\ell^2}{r^2} (1- \ve r^{-\ve})^2 - \frac{\ell}{r^2} (1-\ve r^{-\ve}) (n-2+O(r^{-\alpha}))\right) \mu-\ell \ve^2 r^{-2-\ve}  \right]
\\
& \ \ \ \ + 2 \ell (1-\ve r^{-\ve})  Z v\ \mu.
\notag
\end{align}
With \eqref{eqv2} in hands,  noting that $Z(v^2) = 2 v Zv$, we find from \eqref{k4}
\begin{align}\label{lasterm}
& 2 \int_{\mathcal A} r^{s} e^{r^{-2\ve}} Zv \di(A\nabla v) = -  \int_{\mathcal A} r^{s} e^{r^{-2\ve}} Z(v^2) 
\\
& -  \int_{\mathcal A} r^{s} e^{r^{-2\ve}} Z(v^2) \left(\frac{\ell^2}{r^2} (1- \ve r^{-\ve})^2 - \frac{\ell}{r^2} (1-\ve r^{-\ve}) (n-2+O(r^{-\alpha})) -\ell \ve^2 r^{-2-\ve}\right) \mu
\notag\\
& + 4 \ell \int_{\mathcal A} r^{s+2} e^{r^{-2\ve}} (1-\ve r^{-\ve})  \left(\frac{Z v}{r}\right)^2 \mu
\notag
\end{align}

For the reader's understanding, we briefly describe our next strategy. In the steps \eqref{hj1}, \eqref{hj2} and \eqref{hj4} we use integration by parts to suitably reformulate each term in the right-hand side of \eqref{lasterm}. In addition, one of the resulting terms in \eqref{hj2} is estimated from below using the Cauchy-Schwarz inequality. After choosing the parameter $s$ appropriately, see \eqref{choice1}, the expression \eqref{hj4} is in this way converted into \eqref{hj5}. Combining \eqref{hj1}, \eqref{hj2} and \eqref{hj5}, we finally obtain the bound from below \eqref{kjh1}. With such bound in hands, we return to \eqref{k4}, which we now rewrite by collecting all boundary integrals on one side of the inequality, see \eqref{kjh0}. After one further manipulation, we finally reach the basic inequality \eqref{jkh1}. At that point, a suitable choice of the parameter $\ve$, in dependence of $\alpha$, allows us to establish for sufficiently large values of the inner radius $t$ of the ring $\mathcal A$ in \eqref{om}, and of the parameter $\ell$, the critical positivity bounds \eqref{hog1}, \eqref{ine1}, \eqref{ine5} 

Back to the implementation of the strategy just described. Integrating  by parts, and using (v) in Lemma   \ref{dc1}, we obtain
\begin{align}\label{hj1}
 -  \ine r^{s} e^{r^{-2\ve}} Z(v^2)  =\ &  (n+s) \ine r^s e^{r^{-2\ve}} v^2 - 2\ve \ine r^{s-2\ve} \era v^2  +  \ine r^{s-\alpha} \era v^2\ O(1) 
\\
&  - \ineb  r^{s} \era v^2 \langle Z, \nu \rangle. 
\notag
\end{align}
{\allowdisplaybreaks
Next, we have
\begin{align}\label{hj2}
&-  \int_{\mathcal A} r^{s} e^{r^{-2\ve}} Z(v^2) \left(\frac{\ell^2}{r^2} (1- \ve r^{-\ve})^2 - \frac{\ell}{r^2} (1-\ve r^{-\ve}) (n-2+O(r^{-\alpha})) -\ell \ve^2 r^{-2-\ve}\right) \mu
\\
&= -  \int_{\mathcal A} r^{s} e^{r^{-2\ve}} Z(v^2) \left(\frac{\ell^2}{r^2} (1- \ve r^{-\ve})^2 - \frac{\ell}{r^2} (1-\ve r^{-\ve}) (n-2) -\ell \ve^2 r^{-2-\ve} \right) \mu\notag\\
&+ \ine r^{s-2-\alpha} e^{r^{-2\ve}} \ell (1-\ve r^{-\ve}) vZv\ O(1) 
\notag\\
& \geq  -  \int_{\mathcal A} r^{s} e^{r^{-2\ve}} Z(v^2) \left(\frac{\ell^2}{r^2} (1- \ve r^{-\ve})^2 - \frac{\ell}{r^2} (1-\ve r^{-\ve}) (n-2) -\ell \ve^2 r^{-2-\ve} \right) \mu\notag\\
& - C \ine r^{s-2-\alpha} \era (Zv)^2  - C\ell^2 \ine r^{s-2-\alpha} \era v^2,\notag
\end{align}
where we have used the Cauchy-Schwarz inequality to estimate
\[
\int_{\mathcal A} r^{s-2-\alpha} e^{r^{-2\ve}} \ell (1-\ve r^{-\ve}) vZv\ O(1) \ge - C \int_{\mathcal A} r^{s-2-\alpha} \era (Zv)^2  - C\ell^2 \int_{\mathcal A} r^{s-2-\alpha} \era v^2.
\]
Using the integration by parts formula
\begin{equation}\label{ibpfg}
\int_{\mathcal A} f Zg = \int_{\p \mathcal A} f g \sa Z,\nu\da - \int_{\mathcal A} f g \D Z - \int_{\mathcal A} g Z f
\end{equation}
in the first integral
in the right-hand side of \eqref{hj2}, and (iv) and (v) in Lemma \ref{dc1}, we obtain
\begin{align}\label{hj4}
& -  \int_{\mathcal A} r^{s} e^{r^{-2\ve}} Z(v^2) \left(\frac{\ell^2}{r^2} (1- \ve r^{-\ve})^2 - \frac{\ell}{r^2} (1-\ve r^{-\ve}) (n-2) -\ell \ve^2 r^{-2-\ve} \right) \mu\\
&=\ell^2(n+s-2) \ine r^{s-2} \era (1-\ve r^{-\ve})^2 v^2  \mu -2 \ell^2 \ve \ine r^{s-2-2\ve} \era  (1-\ve r^{-\ve})^2 v^2 \mu\notag\\
& + 2 \ell^2 \ve^2 \ine r^{s-2-\ve} \era  (1-\ve r^{-\ve}) v^2 \mu + \ell^2 \ine r^{s-2-\alpha} \era   (1-\ve r^{-\ve})^2 v^2 O(1)\notag\\
&- \ell^2 \ineb  r^{s-2} \era  (1-\ve r^{-\ve})^2  v^2 \langle Z, \nu \rangle \mu -\ell (n+s-2) (n-2) \ine r^{s-2}\era  (1-\ve r^{-\ve}) v^2 \mu  \notag \\
&+2\ell (n-2)\ve \ine r^{s-2-2\ve} \era  (1-\ve r^{-\ve}) v^2 \mu - \ell (n-2) \ve^2 \ine r^{s-2-\ve} \era v^2 \mu \notag \\
& + \ell (n-2)\ine r^{s-2-\alpha} \era    (1-\ve r^{-\ve})v^2 O(1) +\ell (n-2) \ineb r^{s-2}\era (1-\ve r^{-\ve}) v^2 \langle Z, \nu\rangle \mu \notag\\&-\ell \ve^2 (n+s-2-\ve) \ine r^{s-2-\ve} \era  v^2 \mu + 2 \ell \ve^3 \ine r^{s-2-3\ve} v^2 \mu+\ell \ve^2 \ine r^{s-2-2\ve -\alpha} \era v^2 O(1)  \notag\\
& +\ell \ve^2 \ineb r^{s-2-\ve} \era v^2 \langle Z, \nu \rangle \mu.\notag
\end{align}
The choice 
\begin{equation}\label{choice1} 
s=2-n, 
\end{equation}
eliminates the first and sixth integrals from the right-hand side 
of \eqref{hj4}.
After grouping  the boundary integrals according to like powers of $r$, we obtain 
\begin{align}\label{hj5}
&  -  \int_{\mathcal A} r^{s} e^{r^{-2\ve}} Z(v^2) \left(\frac{\ell^2}{r^2} (1- \ve r^{-\ve})^2 - \frac{\ell}{r^2} (1-\ve r^{-\ve}) (n-2) -\ell \ve^2 r^{-2-\ve} \right) \mu
\\
&= -2 \ell^2 \ve \ine r^{s-2-2\ve} \era v^2 (1-\ve r^{-\ve})^2 \mu + 2 \ell^2 \ve^2 \ine r^{s-2-\ve} \era v^2 (1-\ve r^{-\ve}) \mu
\notag\\
&+ \ell^2 \ine r^{s-2-\alpha} \era (1-\ve r^{-\ve})^2 v^2 O(1)  + 2\ell (n-2)\ve \ine r^{s-2-2\ve} \era v^2 (1-\ve r^{-\ve}) \mu
 \notag\\
&- \ell (n-2) \ve^2 \ine r^{s-2-\ve} \era v^2 \mu + \ell (n-2)\ine r^{s-2-\alpha} \era (1-\ve r^{-\ve}) v^2 O(1)  
\notag\\
 & +\ell \ve^3 \ine r^{s-2-\ve} \era  v^2 + 2 \ell \ve^3 \ine r^{s-2-3\ve} v^2\mu  +\ell \ve^2 \ine r^{s-2-2\ve -\alpha} \era  v^2 O(1)  \notag\\ 
&- \f(\f -n+2) \ineb  r^{s-2} \era v^2 \langle Z, \nu \rangle \mu    + \lt \ineb r^{s-2-\ve} \era  v^2 \langle Z, \nu \rangle \mu
\notag\\ 
& -\f^2 \ve^2 \ineb r^{s-2-2\ve} \era v^2  \langle Z, \nu \rangle \mu.\notag
\end{align}
Substituting \eqref{hj1}, \eqref{hj2}, \eqref{hj5} in \eqref{lasterm}, (keeping \eqref{choice1} in mind), we obtain 
\begin{align}\label{kjh1}
&2 \ine r^{s} \era  Zv \D(A\nabla v)\geq 4 \ell \int_{\mathcal A} r^{s+2} e^{r^{-2\ve}} (1-\ve r^{-\ve})  \left(\frac{Z v}{r}\right)^2 \mu + 2 \ine r^{s} \era v^2  
\\
&  - 2\ve \ine r^{s-2\ve}\era v^2  -2 \ell^2 \ve \ine r^{s-2-2\ve} \era  (1-\ve r^{-\ve})^2 v^2 \mu
\notag\\ 
& + 2 \ell^2 \ve^2 \ine r^{s-2-\ve} \era  (1-\ve r^{-\ve}) v^2 \mu
+2\ell (n-2)\ve \ine r^{s-2-2\ve} \era  (1-\ve r^{-\ve}) v^2 \mu 
\notag\\
& - \ell (n-2) \ve^2 \ine r^{s-2-\ve} \era v^2 \mu
+\ell \ve^3 \ine r^{s-2-\ve} \era  v^2 \mu\notag  + 2 \ell \ve^3 \ine r^{s-2-3\ve} v^2\mu 
\notag\\
& - C \ine r^{s-\alpha} \era v^2 - C \ell^2 \ine r^{s-2-\alpha} \era v^2 -  C \ine r^{s-2-\alpha} \era (Zv)^2
\notag\\
&- \ineb  r^{s} \era v^2 \langle Z, \nu \rangle   - \f(\f -n+2) \ineb  r^{s-2} \era v^2 \langle Z, \nu \rangle \mu\notag\\
& + \lt \ineb r^{s-2-\ve} \era  v^2 \langle Z, \nu \rangle \mu -\f^2 \ve^2 \ineb r^{s-2-2\ve} \era v^2  \langle Z, \nu \rangle \mu. 
\notag
\end{align}
Using \eqref{kjh1} in \eqref{k4} and collecting all   boundary integrals on one side, we find
\begin{align}\label{kjh0}
&  2 \ineb r^{s} \era \left(\frac{Zv}{r}\right)^2   \langle Z, \nu \rangle \mu - \ineb r^{s}\era  \langle A \nabla v, \nabla v\rangle \langle Z, \nu \rangle + \ineb  r^{s} \era v^2 \langle Z, \nu \rangle 
\\
& + \f(\f -n+2) \ineb  r^{s-2} \era v^2 \langle Z, \nu \rangle \mu  - \lt \ineb  r^{s-2-\ve} \era v^2 \langle Z, \nu \rangle \mu
\notag\\
&+ \f^2 \ve^2 \ineb r^{s-2-2\ve} \era v^2  \langle Z, \nu \rangle \mu 
\notag\\
&\geq  2\ve \ine r^{s-2\ve} \era \langle A \nabla v, \nabla v \rangle  -C \ine r^{s-\alpha} \era \langle A \nabla v, \nabla v\rangle 
\notag \\
&+  2s  \ine r^{s-2} \era (Zv)^2 \mu  - 4 \ve \ine r^{s-2-2\ve} \era (Zv)^2 \mu 
\notag\\
& + 4 \ell \ine r^{s-2} \era (Zv)^2 \mu -  4\f \ve \ine r^{s-2-\ve} \era (Zv)^2 \mu 
\notag \\
& + 2 \ine r^{s} \era v^2   - 2\ve \ine r^{s-2\ve} \era v^2 
\notag\\
& -2 \ell^2 \ve \ine r^{s-2-2\ve} \era (1-\ve r^{-\ve})^2 v^2  \mu + 2 \ell^2 \ve^2 \ine r^{s-2-\ve} \era  (1-\ve r^{-\ve}) v^2 \mu
\notag\\
&+2\ell (n-2)\ve \ine r^{s-2-2\ve} \era  (1-\ve r^{-\ve})v^2 \mu - \ell (n-2) \ve^2 \ine r^{s-2-\ve} \era v^2 \mu\notag\\
& +\ell \ve^3 \ine r^{s-2-\ve} \era  v^2 \mu\notag  + 2 \ell \ve^3 \ine r^{s-2-3\ve} v^2\mu - C \ine r^{s-\alpha} \era v^2 
\notag\\
& - C \ell^2 \ine r^{s-2-\alpha} \era v^2 -  C \ine r^{s-2-\alpha} \era (Zv)^2.
\notag
\end{align}
At this point, we add a suitable boundary integral to the left-hand side of \eqref{kjh0}. Such a term is instrumental in achieving, from the positivity of the functional $F$, the conditional positivity of the functional 
$G$ in Lemma \ref{post1}. Specifically, this boundary correction aligns the coercivity properties of 
$F$ with those required for 
$G$, ensuring that the resulting functional retains positivity.
Integrating by parts, and using (v) in Lemma \ref{dc1}, we find
 \begin{align}\label{k50}
&- \ineb r^{s-1} \era v^2  \langle Z, \nu \rangle \mu = -(n+s-1) \ine  r^{s-1} \era v^2 \mu  +2 \ve\ine r^{s-2\ve-1} \era v^2 \mu 
\\
& - \ine r^{s-1-\alpha} \era  v^2 O(1) - 2 \ine  r^{s-1} \era v Zv \mu 
\notag\\
& \geq -(n+s-1) \ine  r^{s-1} \era v^2 \mu  +2 \ve\ine r^{s-2\ve-1} v^2 \mu \era
\notag\\
& - C\ine r^{s-1-\alpha}\era v^2  - \ine r^{s-2} (Zv)^2 \mu \era -\ine r^{s} v^2 \mu\era \notag \\
& = - \ine r^{s-1} \era v^2  \mu  +2 \ve\ine r^{s-2\ve-1} \era v^2 \mu  - C\ine r^{s-1-\alpha}\era v^2 
\notag\\
& - \ine r^{s-2} \era (Zv)^2 \mu  -\ine r^{s} \era v^2 \mu, 
\notag
\end{align}
where in the last equality we have used our choice $s=2-n$, see \eqref{choice1}. Adding the left-hand side of \eqref{k50} to that of  \eqref{kjh0}, and similarly for the right-hand sides, the resulting inequality is
\begin{align}\label{jkh1}
&  2 \ineb r^{s} \era \left(\frac{Zv}{r}\right)^2 \langle Z, \nu \rangle \mu - \ineb  r^{s}\era \langle A \nabla v, \nabla v\rangle \langle Z, \nu \rangle + \ineb  r^{s}\era v^2 \langle Z, \nu \rangle \\
& + \f(\f -n+2) \ineb  r^{s-2} \era v^2 \langle Z, \nu \rangle \mu  - \lt \ineb  r^{s-2-\ve} \era \langle Z, \nu \rangle \mu
\notag \\
&+ \f^2 \ve^2 \ineb r^{s-2-2\ve} \era v^2  \langle Z, \nu \rangle \mu  -  \ineb  r^{s-1}\era v^2 \langle Z, \nu \rangle \mu 
\notag\\
&\geq  2\ve \ine r^{s-2\ve} \era \langle A \nabla v, \nabla v \rangle  -C \ine r^{s-\alpha} \era  \langle A \nabla v, \nabla v\rangle 
\notag  \\
&+  2s  \ine r^{s-2} \era (Zv)^2 \mu  - 4 \ve \ine r^{s-2-2\ve}\era (Zv)^2 \mu 
\notag\\
& +4 \ell \ine r^{s-2} \era (Zv)^2 \mu -  4\f \ve \ine r^{s-2-\ve} \era (Zv)^2 \mu
\notag \\
& + 2 \ine r^{s} \era v^2   - 2\ve \ine r^{s-2\ve}\era v^2 
\notag\\
& -2 \ell^2 \ve \ine r^{s-2-2\ve} \era (1-\ve r^{-\ve})^2 v^2  \mu + 2 \ell^2 \ve^2 \ine r^{s-2-\ve} \era  (1-\ve r^{-\ve}) v^2 \mu
\notag\\
&+2\ell (n-2)\ve \ine r^{s-2-2\ve} \era  (1-\ve r^{-\ve}) v^2 \mu - \ell (n-2) \ve^2 \ine r^{s-2-\ve} \era v^2 \mu
\notag\\
& +\ell \ve^3 \ine r^{s-2-\ve} \era  v^2 \mu
\notag  
+ 2 \ell \ve^3 \ine r^{s-2-3\ve} \era v^2\mu - C \ine r^{s-\alpha} \era v^2 \notag\\
& - C \ell^2 \ine r^{s-2-\alpha} \era v^2 -  C \ine r^{s-2-\alpha} \era (Zv)^2
\notag\\
&  - \ine r^{s-1} \era v^2  \mu  +2 \ve\ine r^{s-2\ve-1} \era v^2 \mu  - C\ine r^{s-1-\alpha}\era v^2 
\notag\\
& - \ine r^{s-2} \era (Zv)^2 \mu  -\ine r^{s}\era v^2 \mu. 
\notag
\end{align}
We now make the following

\medskip

\noindent \underline{Claim:} \emph{With $\alpha$ as in \eqref{st}, suppose that 
\[
2 \ve< \alpha.
\]
Then the right-hand side of \eqref{jkh1} can be made $\ge 0$ provided that $t$ in \eqref{om} and $\ell$ are chosen sufficiently large (depending on $\ve$)}.

\medskip

\noindent Once the claim is proved we can easily complete the proof of the theorem. Observe in fact that the claim implies 
\begin{align}\label{k70}
&  2 \ineb r^{s} \era \left(\frac{Zv}{r}\right)^2  \langle Z, \nu \rangle \mu  - \ineb  r^{s}\era \langle A \nabla v, \nabla v\rangle \langle Z, \nu\rangle  + \ineb  r^{s}\era v^2 \langle Z, \nu \rangle \mu 
\\
& + \f(\f -n+2) \ineb  r^{s-2} \era v^2 \langle Z, \nu \rangle \mu  - \lt \ineb  r^{s-2-\ve} \era v^2 \langle Z, \nu \rangle\mu
\notag \\
&+ \f^2 \ve^2 \ineb r^{s-2-2\ve} \era v^2  \langle Z, \nu \rangle \mu  -  \ineb  r^{s-1} \era v^2 \langle Z, \nu \rangle \mu\ \geq\ 0,
\notag
\end{align}
where $\mathcal A$ is the annulus in \eqref{om}. Noting that, in view of (i) in Lemma \ref{dc1}, we have on $\p \mathcal A$:
\[
\langle Z, \nu \rangle = \sa Z,\nabla r\da = \tau,\ \text{when}\ |x|=\tau,
\ \text{and}\ \langle Z, \nu \rangle =-t,\ \text{ when}\ |x|=t,
\]
and keeping \eqref{choice1} in mind, combining \eqref{k70} with the definition of $F$ in \eqref{g1} we reach the desired conclusion that
\[
\tau^{1-n}F(\ell, \tau) \geq t^{1-n} F(\ell, t). 
\] 
We are thus left with proving the \underline{Claim}. First, once $\ve>0$ has been fixed such that $2\ve< \alpha$, it is clear that 
\begin{align*}
& C \ine r^{s-\alpha} \era \langle A \nabla v, \nabla v\rangle \le C \ine r^{2\ve-\alpha}r^{s-2\ve} \era \langle A \nabla v, \nabla v\rangle
\\
& \le C t^{2\ve-\alpha}\ine r^{s-2\ve} \era \langle A \nabla v, \nabla v\rangle \le 2\ve \ine r^{s-2\ve} \era \langle A \nabla v, \nabla v\rangle,
\end{align*}
provided that $C t^{2\ve-\alpha}\le 2\ve$. This ensures that for $t> t(\ve)>>1$, we have
\begin{align}\label{ine1}
2\ve \ine r^{s-2\ve} \era \langle A \nabla v, \nabla v \rangle  -C \ine r^{s-\alpha} \era \langle A \nabla v, \nabla v\rangle\geq 0.
\end{align}
Next, we estimate from below the sum of those terms in the right-hand side of \eqref{jkh1} containing $(Zv)^2$. Arguing similarly to \eqref{ine1}, one easily sees that for all $\f$ and $t$ large one has
\begin{align}\label{ineq2}
&  2s  \ine r^{s-2}\era (Zv)^2 \mu  - 4 \ve \ine r^{s-2-2\ve}\era (Zv)^2 \mu 
 + 4 \ell \ine r^{s-2} \era (Zv)^2 \mu
 \\
 & -  4\f \ve \ine r^{s-2-\ve} \era (Zv)^2 \mu 
 - C \ine r^{s-2-\alpha} \era (Zv)^2  - \ine r^{s-2} \era (Zv)^2 \mu\  \geq\ 0.
\notag
\end{align}
We next claim that
\begin{align}\label{hog1}
&-2 \ell^2 \ve \ine r^{s-2-2\ve} \era v^2 (1-\ve r^{-\ve})^2 \mu + 2 \ell^2 \ve^2 \ine r^{s-2-\ve} \era v^2 (1-\ve r^{-\ve}) \mu\\
&+2\ell (n-2)\ve \ine r^{s-2-2\ve} \era v^2 (1-\ve r^{-\ve}) \mu - \ell (n-2) \ve^2 \ine r^{s-2-\ve} \era v^2 \mu\notag\\
& +\ell \ve^3 \ine r^{s-2-\ve} \era  v^2 \mu\notag  + 2 \ell \ve^3 \ine r^{s-2-3\ve} \era v^2\mu  
 - C \ell^2 \ine r^{s-2-\alpha} \era v^2\ \geq\ 0,
\notag
\end{align}
provided that $t$ and $\ell$ are chosen sufficiently large.
The crucial positivity of this sum derives from the combination of two facts:
\begin{itemize}
\item[(1)] the observation that, thanks to the fact that $\ve<2\ve<\alpha$, the leading integral is the second one in the left-hand side, i.e., $\ine r^{s-2-\ve} \era v^2 (1-\ve r^{-\ve}) \mu$;
\item[(2)] the fact that the coefficient in front of such integral, $2 \ell^2 \ve^2$, is strictly positive and quadratic in $\ell$.
This allows, for large $\ell$, to dominate the remaining integrals in the sum.
\end{itemize}
Finally, we assert that for all $t$ sufficiently large, we have
\begin{align}\label{ine5}
&2\ine r^{s} \era v^2  -C  \ine r^{s-\alpha} \era v^2 - 2\ve \ine r^{s-2\ve} \era v^2  -  \ine r^{s-1} \era v^2  \mu
\\
&   + 2 \ve\ine r^{s-2\ve-1} \era v^2 \mu  -C \ine r^{s-1-\alpha} \era v^2
-\ine r^{s}\era v^2 \mu\ \geq\ 0.
\notag
\end{align}
This follows by observing that for large $t$, one has by (ii) in Lemma \ref{dc1}
\[
2\ine r^{s} \era v^2 -\ine r^{s}\era v^2 \mu = \ine r^{s}\era v^2 \mu + \ine r^{s-\alpha}\era v^2  O(1).
\]
Since for large $t$ the integral in the right-hand side dominates all remaining integrals in the left-hand side of \eqref{ine5}, the conclusion follows. 
The combination of \eqref{ine1}-\eqref{ine5} establishes the claim, thus completing the proof of the theorem.

}
\end{proof}

We finally turn to the 

 \begin{proof}[Proof of Theorem \ref{mong2}]

Keeping in mind the Definition \ref{D:G} of the functional $G$, with $u$ as in \eqref{mn}, we let  
$w= r^{\ell_1} u$, with $\ell_1=\frac{n-1}{2}$. For ease of notation we henceforth write $u= k(r) w$ with $k(r)= r^{-\ell_1}$. To compute the equation satisfied by $w$ we use \eqref{eqv}, which presently gives
\begin{equation}\label{divw}
\D(A\nabla w)= -w \left(1+  \frac{\D(A\nabla k(r))}{k(r)}\right) - 2 \frac{k'(r)}{rk(r)}  Z w\ \mu.
\end{equation}
Since
\[
\frac{k'(r)}{rk(r)} = - \frac{\ell_1}{r^2} = - \frac{n-1}{2r^2},\ \ \frac{k''(r)}{k(r)} - \frac{k'(r)}{r k(r)} = \frac{\ell_1(\ell_1+2)}{r^2} = \frac{(n-1)(n+1)}{r^2},
\]
we obtain from \eqref{eqk}
\[
\frac{\D(A\nabla k(r))}{k(r)} = \frac{\ell_1(\ell_1+2-n)}{r^2}\mu + \frac{1}{r^{2+\alpha}} O(1)  = -\frac{(n-1)(n-3)}{r^2}\mu + \frac{1}{r^{2+\alpha}} O(1).
\]
Substitution in \eqref{divw} gives
\begin{align}\label{com2}
& \D(A\nabla w)=-w + \frac{(n-1)(n-3)}{4 r^2} w\ \mu   + \frac{n-1}{r^{2}} Zw\ \mu + \frac{w}{r^{2+\alpha}} O(1).
\end{align}
We now multiply both sides of \eqref{com2} by $2 r^{-n+\delta}   Zw$ and integrate over $\mathcal A$, an annular region as in \eqref{om}, obtaining
\begin{align}\label{cot1}
& 2 \ine r^{-n+\delta}   Zw \D(A \nabla w)= -  \ine r^{-n+\delta}  Z (w^2) +  \frac{(n-1)(n-3)}{4} \ine r^{-n-2+\delta} Z(w^2) \mu  
\\
&  + 2(n-1) \ine r^{-n-2+\delta} (Zw)^2  \mu 
+ \ine O(1)  r^{-n-2-\alpha+\delta}   wZw.
\notag \\
& \geq  -  \ine r^{-n+\delta}  Z (w^2) +  \frac{(n-1)(n-3)}{4} \ine r^{-n-2+\delta} Z(w^2) \mu
\notag\\
&  + 2(n-1) \ine r^{-n-2+\delta} (Zw)^2  \mu -C \ine r^{-n-2-\alpha+\delta} (Zw)^2 - C \ine r^{-n-2-\alpha+\delta} w^2, 
\notag
\end{align}
where in the last inequality we have used the trivial estimate $2 |wZw|\le w^2+(Zw)^2$. We next use the integration by parts formula \eqref{ibpfg} in the first and second integrals in the right-hand side of \eqref{cot1}, along with (i), (iv) and (v) in Lemma \ref{dc1}, obtaining
\begin{align*}
& -  \ine r^{-n+\delta}  Z (w^2) = \int_{\mathcal A} w^2 Z(r^{-n+\delta}) + \int_{\mathcal A} r^{-n+\delta} w^2 \D Z - \int_{\p \mathcal A} r^{-n+\delta} w^2 \sa Z,\nu\da
\\
& = (-n+\delta) \int_{\mathcal A} r^{-n+\delta} w^2 + n \int_{\mathcal A} r^{-n+\delta} w^2 + \int_{\mathcal A} O(1) r^{-n-\alpha +\delta} w^2 - \int_{\p \mathcal A} r^{-n+\delta} w^2 \sa Z,\nu\da
\\
& = \delta \int_{\mathcal A} r^{-n+\delta} w^2+ \int_{\mathcal A} O(1) r^{-n-\alpha +\delta} w^2 - \int_{\p \mathcal A} r^{-n+\delta} w^2 \sa Z,\nu\da.
\end{align*}
Similarly, we find
\begin{align*}
& \frac{(n-1)(n-3)}{4} \ine r^{-n-2+\delta} Z(w^2) \mu = \frac{(n-1)(n-3)}{4}\int_{\p \mathcal A} r^{-n-2+\delta} w^2 \sa Z,\nu\da \mu
\\
& - \frac{(n-1)(n-3)}{4} \ine r^{-n-2+\delta} w^2 \D Z\ \mu - \frac{(n-1)(n-3)}{4} \ine w^2 Z(\mu r^{-n-2+\delta})
\\
& = \frac{(n-1)(n-3)}{4}\int_{\p \mathcal A} r^{-n-2+\delta} w^2 \sa Z,\nu\da \mu -  \frac{n(n-1)(n-3)}{4} \ine r^{-n-2+\delta} w^2 \ \mu
\\
& + \int_{\mathcal A} O(1) r^{-n-2-\alpha +\delta} w^2 \mu + \frac{(n+2-\delta)(n-1)(n-3)}{4} \ine r^{-n-2+\delta} w^2  \mu
\\
& +  \int_{\mathcal A} O(1) r^{-n-2-\alpha +\delta} w^2
\\
& = \frac{(n-1)(n-3)}{4}\int_{\p \mathcal A} r^{-n-2+\delta} w^2 \sa Z,\nu\da \mu +  \frac{(2-\delta)(n-1)(n-3)}{4} \ine r^{-n-2+\delta} w^2 \ \mu
\\
& + \int_{\mathcal A} O(1) r^{-n-2-\alpha +\delta} w^2,  
\end{align*}
where in the last equality we have used \eqref{ellmu}.
Substituting the latter two equations in \eqref{cot1}, after rearranging terms, we obtain 
\begin{align}\label{cot22}
& 2 \ine r^{-n+\delta}   Zw \D(A \nabla w)  \ge  2(n-1) \ine r^{-n-2+\delta} (Zw)^2  \mu -C \ine r^{-n-2-\alpha+\delta} (Zw)^2 
\\
& + \frac{(n-1)(n-3)}{4}\int_{\p \mathcal A} r^{-n-2+\delta} w^2 \sa Z,\nu\da \mu - \int_{\p \mathcal A} r^{-n+\delta} w^2 \sa Z,\nu\da
\notag\\
& + \delta \int_{\mathcal A} r^{-n+\delta} w^2+ \int_{\mathcal A} O(1) r^{-n-\alpha +\delta} w^2 
+  \frac{(2-\delta)(n-1)(n-3)}{4} \ine r^{-n-2+\delta} w^2 \ \mu
\notag\\
& + \int_{\mathcal A} O(1) r^{-n-2-\alpha +\delta} w^2   - C \ine r^{-n-2-\alpha+\delta} w^2.  
\notag
\end{align}
At this point we observe that for a fixed $\delta>0$, provided that $t$ in \eqref{om} be sufficiently large, we have
\begin{align*}
& \delta \int_{\mathcal A} r^{-n+\delta} w^2+ \int_{\mathcal A} O(1) r^{-n-\alpha +\delta} w^2 
+  \frac{(2-\delta)(n-1)(n-3)}{4} \ine r^{-n-2+\delta} w^2 \ \mu
\notag\\
& + \int_{\mathcal A} O(1) r^{-n-2-\alpha +\delta} w^2   - C \ine r^{-n-2-\alpha+\delta} w^2\ \ge\ 0.  
\end{align*}
Using this information in \eqref{cot22}, we infer
\begin{align}\label{cot2}
& 2 \ine r^{-n+\delta}   Zw \D(A \nabla w)  \ge  2(n-1) \ine r^{-n-2+\delta} (Zw)^2  \mu -C \ine r^{-n-2-\alpha+\delta} (Zw)^2 
\\
& + \frac{(n-1)(n-3)}{4}\int_{\p \mathcal A} r^{-n-2+\delta} w^2 \sa Z,\nu\da \mu - \int_{\p \mathcal A} r^{-n+\delta} w^2 \sa Z,\nu\da.
\notag
\end{align}
To handle the left-hand side in \eqref{cot2}, we next apply the Rellich-Payne-Weinberger identity in Proposition \ref{P:PW}, with the vector field $X= r^{-n+\delta} Z$, in combination with Lemma \ref{dc1}. We obtain
\begin{align}\label{cot4}
& 2 \ine r^{-n+\delta}  Zw \D(A \nabla w)= 2 \ineb r^{-n-2+\delta} (Zw)^2  \langle Z, \nu \rangle \mu - \ineb r^{-n+\delta}  \langle A \nabla w, \nabla w \rangle \langle Z, \nu \rangle\\    
      & +2(n-\delta) \ine r^{-n-2+\delta} (Zw)^2 \mu  +(\delta-2) \ine r^{-n+\delta} \langle A \nabla w, \nabla w \rangle + \ine O(1) r^{-n+\delta-\alpha}  \langle A \nabla w, \nabla w\rangle.
\notag 
\end{align}
Substituting \eqref{cot4} into \eqref{cot2}, and grouping in the left-hand side the boundary integrals, we conclude that for all $t$ large
\begin{align}\label{cot5}
& 2 \ineb r^{-n-2+\delta} (Zw)^2  \langle Z, \nu \rangle \mu - \ineb r^{-n+\delta}  \langle A \nabla w, \nabla w \rangle \langle Z, \nu \rangle
\\
& + \ineb r^{-n+\delta} w^2 \langle Z, \nu \rangle -\frac{(n-1)(n-3)}{4} \ineb r^{-n-2+\delta} w^2 \langle Z, \nu \rangle    \mu      \notag\\
             & \geq (-2+2\delta)\ine r^{-n-2+\delta} (Zw)^2 \mu + (2-\delta) \ine r^{-n+\delta} \langle A \nabla w, \nabla w\rangle \notag\\
             &  -C \ine r^{-n-2 -\alpha +\delta} (Zw)^2 \mu - C \ine r^{-n-\alpha +\delta } \langle A \nabla w, \nabla w\rangle.
\notag\end{align}      
Next, we note that, with $\delta>0$ fixed, for sufficiently large $t$ we can achieve
\begin{align}\label{lo1}
&   -C \ine r^{-n-2 -\alpha +\delta} (Zw)^2 \mu - C \ine r^{-n-\alpha +\delta } \langle A \nabla w, \nabla w\rangle \\
        &\geq - \frac{\delta}{2}    \left(\ine r^{-n-2 +\delta} (Zw)^2 \mu  + \ine r^{-n +\delta } \langle A \nabla w, \nabla w \rangle\right).\notag
\end{align}
Using \eqref{lo1} in \eqref{cot5}, we finally obtain for large $t$
\begin{align}\label{cot10}
  & 2 \ineb r^{-n-2+\delta} (Zw)^2 \mu \langle Z, \nu \rangle - \ineb \langle A \nabla w, \nabla w \rangle r^{-n+\delta}  \langle Z, \nu \rangle\\
              & + \ineb r^{-n+\delta} w^2 \langle Z, \nu \rangle -\frac{(n-1)(n-3)}{4} \ineb r^{-n-2+\delta} w^2 \langle Z, \nu \rangle  \mu        \notag\\             & \geq (-2+3\delta/2) \ine r^{-n-2+\delta} (Zw)^2 \mu  +(2-3\delta/2) \ine r^{-n+\delta} \langle A \nabla w, \nabla w\rangle\notag \\
& \geq 0,\notag
\end{align}
where in the last inequality we have critically used the inequality \eqref{rad}, which presently gives $$\left(\frac{Zw}{r}\right)^2 \mu \leq \langle A \nabla w, \nabla w\rangle.$$ 
Keeping in mind the definition of $G$ in \eqref{g2}, from the fundamental estimate \eqref{cot10} we finally reach the conclusion of Theorem \ref{mong2}.

\end{proof}             
             

\section{Appendix: Auxiliary material}\label{np}

In this appendix we have collected some material that has been repeatedly used in the main body of the paper. For a given $n\times n$ matrix $M$, we denote by $m_{ij}$ its entries. We routinely identify a vector field $X= (X_1,...,X_n)$ with the first-order differential operator $Xu = \sum_{i=1}^n X_i \frac{\p u}{\p x_i} = \sa X,\nabla u\da$. Consider the open set $\Om = \Rn\setminus \overline B = \{x\in \Rn \mid |x|>1\}$. 
For $x\in \Rn$, we let $r(x)=|x|$. When it is clear from the context that $r$ refers to $|x|$, and not a fixed positive number $r>0$, we simply write $r$ in place of $r(x)$. We record the well-known formula
\begin{equation}\label{lapr}
\Delta r = \frac{n-1}r,\ \ \ \ \ \ \ x\in \Rn\setminus\{0\}.
\end{equation}
As a consequence of \eqref{lapr} and \eqref{st}, we obtain
\begin{equation}\label{Lr}
\D(A(x) \nabla r) = \Delta r + \D(B(x) \nabla r) = \frac{n-1}r\left(1+O(\frac{1}{r^\alpha})\right),\ \ \ \ \ \ \ \ x\in \Om.
\end{equation}
We next introduce the conformal factor $\mu:\Rn\setminus \{0\}\to (0,\infty)$ by the equation
\begin{equation}\label{mu}
\mu(x) = \langle A(x)\nabla r(x), \nabla r(x)\rangle = \frac{\sa A(x)x,x\da}{|x|^2}.
\end{equation}
We clearly have from \eqref{ell}
\begin{equation}\label{ellmu}
\la \le \mu(x) \le \la^{-1}.
\end{equation}
In flat $\Rn$, the radial vector field $X(x) = x$ satisfies $\operatorname{div} X \equiv n$, a property that  plays a crucial role in the analysis of the Laplacian. In the present paper, we work with the following vector field $Z:\Rn \setminus\{0\}\to \Rn$ 
\begin{equation}\label{F}
Z(x) = r(x) \frac{A(x) \nabla r(x)}{\mu(x)}.
\end{equation}
 We note explicitly that, when $A(x) \equiv I_n$, then $Z(x) = x$. We also observe that, given a function $f$, the Cauchy-Schwarz inequality and the definition \eqref{mu} imply the following
estimate 
\begin{equation}\label{rad}
\left(\frac{Zf}{r}\right)^2 \mu \le \sa A(x)\nabla f,\nabla f\da.
\end{equation}
The validity of \eqref{rad} is an easy consequence of the Cauchy-Schwarz inequality
 \begin{align*}
\left( \frac{Zf}{r}\right)^2 \mu & = \frac{ \langle A \nabla f, \nabla r \rangle^2}{\mu}= \frac{ \langle A^{1/2} \nabla f, A^{1/2} \nabla r\rangle^2}{\mu} \leq \frac{\langle A^{1/2} \nabla f, A^{1/2} \nabla f \rangle \langle A^{1/2} \nabla r, A^{1/2} \nabla r \rangle}{\mu}
\\
& = \frac{\langle A\nabla f, \nabla f\rangle \langle A \nabla r, \nabla r\rangle}{\mu} = \langle A \nabla f ,\nabla f\rangle.
\end{align*} 
Furthermore, when $f(x) = h(|x|)$, then \eqref{F} gives
\begin{equation}\label{Zf}
A \nabla f = \frac{h'(r)}{r} \mu Z,
\end{equation}
and therefore if $u$ is another function, we find
\begin{equation}\label{Zuf}
\sa A\nabla u,\nabla f\da = \frac{h'(r)}{r} \mu Zu.
\end{equation} 
The vector field \eqref{F} has several remarkable properties. In the following lemma we collect those that will be important to us. In the sequel, we use the notation $D_i = D_{x_i}, D_{ik} = D_{x_i x_k}$, etc.

\begin{lemma}\label{dc1}
One has in $\Rn\setminus\{0\}$
\begin{itemize}
\item[(i)] $Zr= r$,
\item[(ii)] $\mu= 1+ O(1/r^{\alpha})$,
\item[(iii)] $|\nabla \mu|= O(1/r^{1+\alpha})$,
\item[(iv)] $|Z\mu| = O(1/r^{\alpha})$,
\item[(v)] $\D Z= n+O(1/r^{\alpha})$,
\item[(vi)] $\left|[D_i,  Z]u -D_i u \right| \leq \frac{C_2}{r^{\alpha}} |\nabla u|,\ \ \ \forall i \in \{1, ...,n\}$.
\end{itemize}
\end{lemma}

\begin{proof}
(i) We have 
\begin{align}
& Zr =r \frac{\langle A\nabla r, \nabla r\rangle}{\mu} =r \frac{\langle A\nabla r, \nabla r\rangle}{\langle A \nabla r, \nabla r\rangle}  = r.
\end{align}

\medskip

(ii) Keeping \eqref{B} in mind, we have 
\begin{align}
\mu= \langle \nabla r, \nabla r\rangle  + \langle B \nabla r, \nabla r\rangle = 1+ O(1/r^\alpha), 
\end{align}
where in the last equality, we have used the decay assumption \eqref{st}. 

\medskip

(iii) Using summation convention, we have 
\[
D_k \mu = D_k(a_{ij} D_i r D_j r) = D_k(a_{ij}) D_i r D_j r + 2 a_{ij} D_{ik} r D_j r. 
\]
Since \eqref{st} gives $D_k(a_{ij}) = O(r^{-(1+\alpha)})$ and since $D_{ik} r = O(r^{-1})$, the conclusion follows.

\medskip

(iv)  Applying \eqref{rad} with $f = \mu$, we find
\[
\left(\frac{Z\mu}{r}\right)^2 \mu \le \sa A(x)\nabla \mu,\nabla \mu\da\le \la^{-1} |\nabla \mu|^2,
\]
where in the second inequality we have used \eqref{ell}. The desired conclusion now follows from (iii) and \eqref{ellmu}.

\medskip

(v) We have from \eqref{F} and \eqref{Lr}
\begin{align*}
 \D Z & = r \mu^{-1} \D(A(x)\nabla r) + 1 - r \mu^{-2} \sa A(x)\nabla r,\nabla\mu\da
\\
& = (n-1) \left(1 + \frac{1-\mu}{\mu}\right) \left(1+O(\frac{1}{r^\alpha})\right)+ 1 - r \mu^{-2} \sa A(x)\nabla r,\nabla\mu\da.
\end{align*}
From (ii) we have $1 + \frac{1-\mu}{\mu} = 1 + O(r^{-\alpha})$, and therefore
\[
\left|r \mu^{-2} \sa A(x)\nabla r,\nabla\mu\da\right|\le C(\la) |\nabla r| |\nabla \mu| = O\left(\frac{1}{r^{1+\alpha}}\right),
\]
where we have used \eqref{ellmu} and (iii). Substituting in the above expression of $\D Z$, we obtain (v).

\medskip

(vi) Using \eqref{B}, with $\omega(x) = \mu(x)^{-1}(1-\mu(x)) = O(r^{-\alpha})$, we write
\[
Z = \mu^{-1} (x + B x) = (1 + \omega(x))(x + B x) = x + \omega(x) x + \omega(x) B x.
\]
Keeping in mind that, with $Y(x) = x$, we have
\[
[D_i,Y] = D_i,
\]
we easily find
\[
[D_i,Z]u - D_i u = \omega D_i u + D_i \omega Yu + \omega [D_i,BY]u + D_i \omega BYu.
\]
Since by \eqref{ellmu} and (iii) we have
\[
D_i \omega = O(r^{-\alpha}),
\]
the desired conclusion follows.

\end{proof}
We note explicitly that (i) of Lemma gives for any $\gamma\in \R\setminus\{0\}$
\begin{equation}\label{Zrad}
Z(r^\gamma) = \gamma\ r^{\gamma}.
\end{equation}

The following lemma plays a crucial role in this paper.

\begin{lemma}\label{L:rad}
Let $u$ and $v$ be two functions satisfying $v(x) = h(|x|) u(x)$.
Then we have
\begin{equation}\label{rad1}
\left(\left(\frac{Zv}r\right)^2 \mu - \sa A(x)\nabla v,\nabla v\da\right) = h(r)^2\left(\left(\frac{Zu}r\right)^2 \mu - \sa A(x)\nabla u,\nabla u\da\right).
\end{equation}
In particular, when $v(x)=h(|x|)$, we obtain from \eqref{rad1}
\begin{equation}\label{rad2}
\left(\frac{Zv}{r}\right)^2 \mu = \sa A(x)\nabla v,\nabla v\da.
\end{equation}
\end{lemma}

\begin{proof}
We have from (i) in Lemma \ref{dc1}
\[
\frac{Z v}{r} = h'(r)u + \frac{h(r)}r Z u.
\] 
This gives
\begin{align*}
\left(\frac{Z v}{r}\right)^2 \mu & = \left(h'(r)u + \frac{h(r)}r Z u\right)^2 \mu = h'(r)^2 u^2 \mu+ h(r)^2\left(\frac{Z u}r\right)^2 \mu + 2 h(r) h'(r) u \frac{Zu}r \mu 
\end{align*}

On the other hand
\begin{align*}
\sa A\nabla v,\nabla v\da & = h'(r)^2 u^2 \mu + 2 h(r) h'(r) u \sa A\nabla u,\nabla r\da + h(r)^2 \sa A\nabla u,\nabla u\da
\\
& = h'(r)^2 u^2 \mu + 2 h(r) h'(r) u \frac{Zu}r \mu + h(r)^2 \sa A\nabla u,\nabla u\da,
\end{align*}
where in the second equality we have used \eqref{F}. The identity \eqref{rad1} now follows by subtracting the latter two identities. Finally, by taking $u\equiv 1$ in \eqref{rad1}, we immediately obtain the identity \eqref{rad2}.

\end{proof}

We need the following result from \cite{PW}, in the version given in \cite[Lemma 2.11]{GV}.

\begin{prop}[Identity of Rellich-Payne-Weinberger]\label{P:PW}
Let $U\subset \Rn$ be a piecewise $C^1$, bounded open set,   
and let $X = (X_1,...,X_n) \in C^{0,1}(\overline U,\Rn)$ be a Lipschitz vector field. If $f\in C^{1,1}(\overline U)$, then
\begin{align*}
& 2 \int_U Xf \di(A\nabla f) dx =
 2 \int_{\p U} Xf \sa A\nabla f,\nu\da d\sigma - \int_{\p U} \sa A\nabla f,\nabla f\da \sa X,\nu \da d\sigma 
\\
& + \int_{U} \di X \sa A\nabla f,\nabla f\da dx 
 - 2 \int_U a_{ij} [D_i, X] f D_j f dx + \int_U Xa_{ij} D_i f D_j f dx.
\end{align*}
\end{prop}

Note that, when $A(x) = I$ and $X(x) = x$, Proposition \ref{P:PW} reduces to the well-known identity of Rellich in \cite{Relid}
\begin{align*}
\int_{\p U} |\nabla f|^2 \sa x,\nu \da d\sigma  = (n-2) \int_{U} |\nabla f|^2 dx 
  + 2 \int_{\p U} \sa\nabla f,x\da \frac{\p f }{\p \nu} d\sigma - 2 \int_U \sa\nabla f,x\da \Delta f dx.
\end{align*}

                                             
\vskip 0.3in

\section{Declarations}

\noindent \textbf{Data availability statement:} This manuscript has no associated data.

\vskip 0.2in

\noindent \textbf{Funding and/or Conflicts of interests/Competing interests statement:} The authors declare that they do not have any conflict of interest for this work

\

\end{document}